\documentclass[12pt, reqno]{amsart}
\usepackage{amsmath, amsthm, amscd, amsfonts, amssymb, graphicx, color}
\usepackage[bookmarksnumbered, colorlinks, plainpages]{hyperref}

\makeatletter
\@namedef{subjclassname@2010}{%
	\textup{2010} Mathematics Subject Classification}
\makeatother
\usepackage[mathscr]{euscript}
\usepackage{latexsym}
\usepackage{multicol,tabularx}
\usepackage{tikz}
\usepackage{graphics}
\usepackage{mathtools}

\usepackage{amsfonts}
\usepackage{amssymb}
\usepackage[english]{babel}
\usepackage{multido}
\usepackage[nomessages]{fp}
\usepackage[margin=2.5cm]{geometry}
\usepackage{enumitem}
\newtheorem{thm}{Theorem}[section]
\newtheorem{prop}{Proposition}[section]

\newtheorem{lem}{Lemma}[section]

\newtheorem{defn}{Definition}[section]
\newtheorem{rmk}{Remark}

\numberwithin{equation}{section}
\begin{document}

	\setcounter{page}{1}
	\title[]{Normal categories of semigroup of order-preserving transformations on a finite chain}
	\begin{abstract}
	K. S. S. Nambooripad intoduced nornal categories to enable to describe the structure of regular semigroups fully.  In this paper we describe the ideal categories of the regular semigroup $OX_n,$ of non-invertible order-preserving transformations on a finite chain $X_n=\{1\leq 2\leq \cdots \leq n \}$ which are normal categories.
	 Further it is shown that the principal left ideal category  of $OX_n$ as the  power set category $P_o(X_n)$ of $OX_n$ and the principal right ideal category as $\prod_o(X_n)$ category of ordered partitions of $X_n$ and described the cone semigroup $T\mathscr L(OX_n)$ and prove that it  is isomorphic to $OX_n.$

	\end{abstract}
	\author[K. K. Sneha, P. G. Romeo]{K. K. Sneha$^{1}$,     P. G. Romeo$^{2}$}
	
	\address{$^{1}$ Department of Mathematics, Cochin University of Science And Technology, Kerala, India.}
	\email{romeo\_ parackal@yahoo.com
	}
	
	\address{$^{2}$ Department of Mathematics, Cochin University of Science And Technology, Kerala, India.}
	\email{snehamuraleedharan0007@gmail.com
	}
	\thanks{The first author wishes to acknowledge the financial support of the Council for Scientific and
		Industrial Research, New Delhi (via JRF and SRF) in the preparation of this article.}


	\date{Received: xxxxxx; Revised: yyyyyy; Accepted: zzzzzz.
		\newline \indent $^{1}$ Corresponding author}
	
	\baselineskip=21pt
	\vspace*{-1 cm}
	\maketitle
	\section{Introduction}
		There are mainly two established procedures for analyzing the structure of  arbitrary regular semigroups. The first approach was initiated by W. D. Munn via fundamental inverse semigroups during  1966-1970. The idempotent elements are considered as basic objects and based on which the structure theory has developed. From this perspective, K. S. S. Nambooripad introduced the notion of a regular biordered set in the celebrated work\cite{ams} and proved that the set of idempotents of a regular semigroup is a regular biordered set.
		Later on, Easdown proved that all biordered sets arise as biordered sets of semigroups \cite{easdown}. Nambooripad's theory of biordered sets played a significant role in the structure theory of arbitrary semigroups.
		\par The second approach was initiated by T. E. Hall. In this approach, the ideals of the semigroup are the fundamental object and based on the relation that existed between the left and right ideals of the semigroup, the structure theory has been build. Hall proved that, for a regular semigroup $S$, there is a representation of $S$ by pairs of order-preserving transformations on the partially ordered sets $I=S/\mathscr R$ and $\Lambda= S/\mathscr L$ where $\mathscr L\text{ and }\mathscr R$ are Green's relations on $S.$ He showed that this  representation is faithful if and only if $S$ is a fundamental regular semigroup. Moreover, Hall completely identified the maximal fundamental regular semigroup associated with $S.$	Later on, P. A. Grillet refined Hall's theory by introducing the notion of cross-connections of partially ordered sets and characterized the partially ordered sets involved as regular partially ordered sets. Grillet defines a cross-connection between two regular partially ordered sets $I$ and $\Lambda$ as a pair of order-preserving mappings 
		$$ \Gamma: I\rightarrow \Lambda^* \textrm{ and }\Delta: \Lambda\rightarrow I^*$$ 
		satisfying certain axioms, where $I^*[\Lambda^*]$ denotes the dual of $I[\Lambda].$  In \cite{pag}, Grillet explicitly described the relationship between the partially ordered sets of the principal ideals of a given regular semigroup under inclusion. Moreover,
		Grillet shows that a pair of regular partially ordered sets arises from a regular semigroup $S$ as $S/\mathscr L$ and $S/\mathscr R$ if
		only if there is a cross-connection between them.  One of the significant	limitations of this method is that a regular semigroup and its full regular subsemigroup induce the same cross-connection. This shows that one can not obtain  the complete semigroup. Later, Nambooripad  identified this limitation and observed that the data provided by partially ordered sets are insufficient to characterize a arbitrary regular semigroups completely. 
		\par 
		In the early 1990's Nambooripad generalized the cross-connection theory by considering certain categories, which he called normal categories, and described the structure of regular semigroups \cite{kn}. These are the categories of principal left ideals $\mathscr L(S)$ and principal right ideals $\mathscr R(S)$ of a regular semigroup $S$.  A cross-connection between two normal categories $\mathcal C$ and $\mathcal D $ is a local isomorphism $ \Gamma : \mathcal D\to N^*\mathcal C$ where
		$N^*\mathcal C$ is the normal dual of the category $\mathcal C$. A cross-connection $\Gamma$ determines a cross-connection
		semigroup; conversely, every regular semigroup is isomorphic to a cross-connection semigroup for a suitable cross-connection. 
		\par
		Nambooripad's theory is very abstract and includes many technical details.  Currently  many researchers are working in this area and many research articles 
		are being published  (See.\cite{com},\cite{pa},\cite{am},\cite{ah},\cite{pz},\cite{pgr},\cite{az},\cite{np}).
		 This paper, considers studying  the structural properties of the semigroup $OX_n$ of order-preserving transformations  on a finite chain $X_n.$ 
		\par
		
		This article is organized as follows. In section \ref{s2}, we discuss the important concepts and results regarding the general theory of cross-connection, proposed by K. S. S. Nambooripad and also the structural properties of the semigroup $OX_n$. In section \ref{s4}, we characterize the principal left ideal category. 
		
\section{Preliminaries }\label{s2}
In the sequel, we recall the basic concepts and definitions in the general theory of cross-connection, introduced by K.S.S.Nampooripad. We assume  familiarity with the basic concepts in semigroup theory and category theory, for a detailed discussion see :\cite{ah}\cite{pag}\cite{hall}\cite{jmh} for semigroup theory and \cite{Mac} for basic category theory.	Throughout this paper, we write transformation to the right of their argument and take the composition from left to right. \par
  
For  category $\mathcal C$,
$v\mathcal C$ denotes the set of objects of $\mathcal C$ and $\mathcal C(a,b)$ the morphisms from $a$ to $b$. We  assume that categories under consideration are categories with subobjects. 
%A preorder $\mathcal P$ is a category such that for any $ p, p'\in v\mathcal P$, the hom-set $\mathcal P(p, p')$ contains at most one morphism. In this case, the relation $\subseteq$ on the class $v\mathcal P$ of objects of 
%$\mathcal P$ is defined by $p\subseteq p'$ if $\mathcal P(p,p') \neq \emptyset$
%is a quasi- order and $\mathcal{P}$ is said to be a strict preorder if $\subseteq$ is a partial order. 
%\begin{defn}\emph{(\cite{az}\cite{kn})} Let $\mathcal C$ be a small category and $\mathcal P$ be a subcategory of $\mathcal C$ such that 
%	$\mathcal P$ is a strict preorder with $v\mathcal P= v\mathcal C.$ Then $(\mathcal C, \mathcal P)$ is a $category\; with\; sub\;objects$ if 
%	\begin{enumerate}[label=\emph{(\arabic*)}]
%		\item every $f\in \mathcal P$ is a monomorphism in $\mathcal C$
	%	\item if $f=hg$ for $f,g\in \mathcal P,$ then $h\in \mathcal P.$
%	\end{enumerate}
%\end{defn}
In a category $(\mathcal C, \mathcal P)$ with subobjects, morphisms in $\mathcal P$ are  called inclusions. If $c'\to c$ is an inclusion, we write $c'\subseteq c$ and denotes this inclusion by $j^c_{c'}$. An inclusion $j^c_{c'}$ splits if there exists $q:c\to c'\in \mathcal C$ such that $j^c_{c'}q = 1_{c'}$ and the morphism $q$ is called a retraction. A normal factorization of a morphism $f\in \mathcal C(c,d)$ is a factorization of the form $f=quj$ where $q:c\to c'$ is a retraction, $u: c'\to d'$ is an isomorphism and $j= j^d_{d'}$ is an inclusion where $c', d'\in v\mathcal C$ with $c'\subseteq c,\; d'\subseteq d$. The morphism $qu$ is known as the epimorphic component of $f$ and is denoted by $f^{\circ}$.

%Let $\mathcal C\text{ and }\mathcal D$ be categories with subobjects. A functor $F:\mathcal %C\to\mathcal D$ is  a local isomorphism if $F$ is inclusion preserving, fully-faithful and for %each $c\in v\mathcal C,\; F\vert \langle c\rangle $ is an isomorphism of the ideal $\langle %c\rangle$ onto $\langle F(c) \rangle$. Here the ideal $\langle c\rangle$ is the full subcategory %of $\mathcal C$ with object set $\{a\in v\mathcal C: a\subseteq c\}.$

\begin{defn}\emph{(\cite{az}\cite{kn})\label{d1}}
	Let $\mathcal C$ be a category with subobjects and $d\in v\mathcal C$.  A map $\gamma: v\mathcal C \to \mathcal C$ is called a  cone from the base $v\mathcal C$ to the vertex $d$ if 
	\begin{enumerate}[label=\emph{(\arabic*)}]
		\item $\gamma(c)\in \mathcal C(c,d)$ for all $c\in v\mathcal C$
		\item if $c\subseteq c'$ then $j^{c'}_c\gamma(c') = \gamma(c)$
	\end{enumerate}
\end{defn}

For a  cone $\gamma$ denote by $c_\gamma$ the vertex of $\gamma$ and for $c\in v\mathcal C$, the morphism $\gamma(c): c\to c_\gamma$ is called the component of $\gamma$ at $c$.
A cone $\gamma$ is said to be normal if there exists $c\in v\mathcal C$ such that $\gamma(c): c\to c_\gamma$ is an isomorphism. We denote by $T\mathcal C$ the set
of all normal cones in 
$\mathcal C$ and by  $M_\gamma ,$ the set 
$$ M_\gamma =\{c\in v\mathcal C: \gamma(c)\text{ is an isomorphism } \}.$$
\begin{defn}\emph{(\cite{az}\cite{kn})}
	A category $\mathcal C$ with subobjects is called a normal category if the following holds
	\begin{enumerate}[label=\emph{(\arabic*)}]
		\item any morphism in $\mathcal C$ has a normal factorization
		\item every inclusion in $\mathcal C$ splits
		\item for each $c\in v\mathcal C$ there is a normal cone $\gamma $ with vertex $c$ and $\gamma(c) = 1_{c_\gamma}.$
	\end{enumerate}
\end{defn}
Observe that given a normal cone $\gamma$ and an epimorphism $f: c_\gamma \to d$ the map $\gamma * f : a \to \gamma(a)f$ from $v\mathcal C$ to $\mathcal C$ is a normal cone with vertex $d$. Consider two normal cones 	$\gamma$ and $\sigma$, then 
$$\gamma\cdot \sigma = \gamma*(\sigma(c_\gamma))^{\circ} $$ 
where $(\sigma(c_\gamma))^{\circ}$ is the epimorphic part of $\sigma(c_\gamma)$, defines a binary composition on $T\mathcal C$.
\begin{thm} \emph{(Theorem III.2 \cite{kss})}\label{e1}
	Let $\mathcal C$ be a normal category. Then $T\mathcal C$ the set of all normal cones in $\mathcal C$ is a regular semigroup with the binary operation 
	\begin{equation}\label{e1}
		\gamma\cdot\sigma = \gamma*(\sigma(c_\gamma))^\circ
	\end{equation} 
	and $\gamma\in T\mathcal C$ is idempotent if and only if $\gamma(c_{\gamma})=1_{c_{\gamma}}$.
\end{thm}    
%Now we provide the definition of cross connection and the description of the semigroup of a cross connection following \cite{muhammed2019inductive} and  \cite{rajan2022cross}.

\par
The two normal categories are associated with a regular semigroup $S,$ namely the principal left ideal category $\mathscr L(S)$ and the principal right ideal category $\mathscr R(S).$ The objects of $\mathscr L(S)$ are principal left ideals $Se$ generated by idempotents $e\in E(S).$	The morphisms are partial right translations $\rho_u :Se\to Sf: u \in eSf.$ Dually, the objects of the category	$\mathscr R(S)$ of principal right ideals are $eS,$ and the morphisms are partial left translations $\lambda_v : v\in fSe.$ They have
an obvious choice of subobjects, namely the preorder induced by set inclusions. 	%Further, if we do the cross-connection construction 	with the categories $\mathscr L(S)$ and $\mathscr R(S)$ then $S(\Gamma,\Delta)$ will be isomorphic to $S$. This will give a faithful representation of the
%semigroup $S$ as a subsemigroup of $T\mathscr L(S) \times (T\mathscr R(S))^{op},$ being a sub-direct product.
\begin{prop}\emph{(\cite{kn}\cite{com})}\label{p21}		
	
	Let $S$ be a regular semigroup. Then $\mathscr L(S)$ is a normal category. $\rho (e, u, f) = \rho (e', v, f')$ if and only if $e\mathscr L e', f \mathscr L f', u \in eSf, v \in e'Sf' \text{ and } v = e'u.$ Let
	$\rho = \rho(e, u, f)$ be a morphism in $L(S)$. For any $g \in  \mathscr R_u \cap  \omega(e) \text{ and } h \in  E(\mathscr L_u),
	\rho = \rho(e, g, g)\rho(g, u, h)\rho(h, h, f)$ 
	is a normal factorization of $\rho.$
\end{prop}

\begin{prop}\emph{(\cite{kn}\cite{com})}\label{p22}
	Let $S$ be a regular semigroup, $a \in S$ and $f \in E(\mathscr L_a).$ Then for each $e \in E(S),$ let $\rho^a(Se) = \rho(e, ea, f)$. Then $\rho^a$
	is a normal cone in $ \mathscr L(S)$ with vertex $Sa$ called
	the principal cone generated by $a$. $$M\rho^a = \{Se : e \in E(\mathscr R_a)\}.$$ 
	$\rho^a$
	is an idempotent in $T\mathscr L(S)$ iff
	$a \in E(S)$. The mapping $a \mapsto \rho^a $
	is a homomorphism from $S$ to $T\mathscr L(S)$. Further if $S$ is a monoid,
	then $S$ is isomorphic to $T\mathscr L(S)$.
\end{prop}

\subsection{Semigroup of order-preserving transformations  on a finite chain}\label{s3}
%In this section we discuss the Greens relations in the semigroup $OX_n$ of singular order-preserving transformations  on a finite chain

Let $X_n=\{1\leq 2\leq\cdots\leq n: n\in \mathbb N\}$ be a finite chain. A transformation $f:X_n\to X_n$ is called order-preserving if $(i)f\leq (j)f$ whenever $i\leq j.$  The semigroup of all singular order-preserving mappings from $X_n$ to itself under function composition is denoted by $OX_n$. 
In order to avoid trivialities, assume $n\geq 3.$   The Green's relations in the semigroup $OX_n$ are charcterized entirely in terms of their images and kernels. It is known that $OX_n$ is a regular subsemigroup of $TX_n,$ the full transformation semigroup of $X_n.$ The following lemma characterize all the Green's equivalences in $OX_n.$
\begin{lem}\label{le1}\cite{gracinda}
	Let $f$ and $g$ be elements of the semigroup $OX_n$ of order-preserving transformations  on a finite chain $X_n.$ Then we have the following.
	\begin{enumerate}[label=\emph{(\arabic*)}]
		\item $f\;\mathscr R\; g$ if and only if $ ker\; g =ker  \;f, $
		\item $f\; \mathscr L \;g $ if and only if $ Im \; f = Im  \;g, $
		\item $f\; \mathscr H \;g $ if and only if $ f = g, $
		\item $f \mathcal J g $ if and only if $ |Im \; f| = |Im  \;g|. $
		
	\end{enumerate}
\end{lem}
\begin{rmk}
	Since the Green's $\mathscr H$ relation in $OX_n$is identity, the semigoup $OX_n$ is a fundamental regular semigroup which is a subsemigroup of full transformation semigroup of $X_n.$ 
	
\end{rmk}

	\section{ The category of Principal left ideals of $OX_n$}\label{s4}
	In this section we characterize the normal category $\mathscr L(OX_n)$ associated with principal left ideals of $OX_n$. Here we use $S$ and $OX_n$ mutually to denote the semigroup of non-invertible transformations on $X_n.$ For any proper nontrivial subset $A$ of $X_n, $  let $e_A$ denote the idempotent transformation with image $A.$ Note that $e_A$ is not uniquely determined by $A.$

 %We will use $SA$ to denote the principal left ideal of $OX_n$ generated by $e_A$. It may be noted that even though $e_A$ is not uniquely defined by $A,$ the notation $SA$ is well-defined.
	\begin{lem}\label{l1}
		Let $A,B \subsetneq X_n$ and $\rho(e_A, u, e_B)$ be a morphism from $Se_A$ to $Se_B$. Then for any $x\in A,\; xu \in B.$ Also $\rho(e_A,u,e_B) =\rho(e'_A,v,e'_B)$ if and only if $xu=xv$ for all $x\in A,$ where $e_A,\;e'_A$ are idempotents with image $A$ and $e_B,\;e'_B$ are idempotents with image $B.$
	\end{lem}
	\begin{proof}
		
		By the definition of a morphism in $\mathcal L(OX_n)$, $u\in e_ASe_B$ and $Xu\subseteq Xe_B= B.$ In particular  $xu \in B$ for all $x\in A.$ To prove the second assertion, let   $\rho(e_A,u,e_B) =\rho(e'_A,v,e'_B),$ then by Proposition \ref{p21} $ u=e_Av.$ Also since $e_A$ is an idempotent map with image $A$ it can be seen that $e_A\vert_A = 1_A$. Hence $xu = xv$ for all $x\in A.$ Conversely if $xu = xv$ for all $x\in A,$ then since $u\in e_ASe_B,$ $e_Au = u$ and by our assumption $e_Au= e_Av.$ Hence $u=e_Av$ and using Proposition \ref{p21} we have $\rho(e_A,u,e_B) =\rho(e'_A,v,e'_B).$
	\end{proof}
	\begin{prop}\label{p1}
		All normal cones in the category $\mathscr L(OX_n)$ are principal cones.
		
	\end{prop}
	\begin{proof}
		Let $\gamma$ be a normal cone in $\mathscr L(OX_n)$, with $c_\gamma =Se_A$ for some $e_A \in E(OX_n).$ For any $x\in X_n,\; e_x$ denotes the constant map whose image is $x$ and $Se_x=\{x\}$. Consider $\gamma(Se_x) $ for $x\in X.$ Let $\gamma(Se_x) =\rho(e_x,u_x,e_A)$. By Lemma \ref{l1} $xu_x\in A.$ Since $\gamma(Se_x)$ is uniquely determined by $x$, $u_x$ is uniquely determined by $x.$ Define $\alpha$ on $X_n$ as follows. $$ x\alpha=xu_x\;\;\text{ for all } x\in X_n \text{ and }u_x \text{ as above }.$$ Since $u_x$ is uniquely determined by $x,\;\alpha $ is well defined. Since $xu_x\in A$ for all $x\in X_n,$ $\alpha$ is a function from $X_n$ with image contained in $A.$ Now we prove that $\alpha$ is an order-preserving transformation.  For, if possible assume that $\alpha$ is not an order-preserving function. Then there exists $x,y\in X_n$ such that $x\alpha\leq y\alpha$ for $x\geq y.$ Now consider the set $Y=\{x,y\}$ such that $Sx,Sy\subseteq Se_Y$ and $\gamma(Se_Y)=\rho(e_Y, u, e_A)$. Since $Sx,Sy\subseteq Se_Y$ we have
		$$\gamma(Se_x)=j^{Se_Y}_{Se_x}\gamma(Se_Y)\text{ and } \gamma(Sy)=j^{Se_Y}_{Sy}\gamma(Se_Y).$$ That is $$\rho(e_x,u_x,e_A) =\rho(e_x, e_x,e_Y)\rho(e_Y, u,e_A)=\rho(e_x, e_xu,e_A).$$  Similarly we get $\rho(e_y,u_y,e_A) =\rho(e_y,e_yu,e_A).$ From these two equations we get $x\alpha =xu_x = xe_xu =xu$ and $y\alpha =yu.$ Hence $xu\leq yu $ for $x\geq y,$ which is a contradiction to the fact that $u$ is order-preserving. Therefore $\alpha \in OX_n$ is an order-preserving transformation with image $\alpha$ contained in $A.$ Since $\gamma$ is a normal cone there is a component $\gamma(Se_C)$ is an isomorphism and let $\gamma(Se_C) =\rho(e_C, \beta,e_A)$. Then by Lemma \ref{l1} $x\beta \in A$ for all $x\in C.$ Since $\gamma(Se_C)$ is an isomorphism $\beta\;\mathscr L\; e_A,$ and $Im\; \beta = A.$ Now we show that $Im\; \alpha = A.$ Let $y
		\in A,$ then there exists $x\in C$ such that $x\beta =y.$
		$$ \rho(e_x, u_x, e_A) =\gamma(Sx) =j^{Se_C}_{Sx}\gamma(Se_C) = \rho(e_x,e_x\beta, e_A ).$$ Thus $u_x= e_x\beta$(using Proposition \ref{p21}), so that $x\alpha = xu_x = xe_x\beta= x\beta=y.$ Hence $\alpha$ is onto. Now we prove that $\gamma =\rho^\alpha.$ Since $Im\;\alpha= Se_A$ the vertex of $\rho^\alpha$ is $Se_A=c_\gamma.$ For $B\subseteq X,$ we prove that, if $\gamma(Se_B)=\rho(e_B,v,e_A)$ , then $\rho(e_B,v,e_A) = \rho(e_B,e_B\alpha,e_A)$. For that it is sufficient to prove that $xv=xe_B\alpha$ for all $x\in B.$ If $x\in B,$ then $Sx\subseteq Se_B$ and by the definition of cones $$ \gamma(Sx) = j^{Se_B}_{Sx}\gamma(Se_B) = \rho(e_x,e_x,e_B)\rho(e_B,v,e_A )=\rho(e_x,e_xv,e_A).$$ But $\gamma(Sx) =\rho(e_x,u_x,e_A),$ equating these we get $xu_x =xe_xv =xv.$ That is for all $x\in B$ we have $ x\alpha = xv.$ Therefore $\rho(e_B,v,e_A) =\rho(e_B,e_B\alpha,e_A).$ Hence $\gamma =\rho^\alpha$ and all normal cones are of the form $\rho^\alpha$ for some $\alpha\in OX_n.$
	\end{proof}
	\begin{thm}\label{t1}
		The semigroup of normal cones in $\mathscr L(OX_n)$ is isomorphic to $O(X_n).$
	\end{thm}
	\begin{proof}
		It is known that, the map $ \phi:T\mathscr L(OX_n)\to \mathscr L(OX_n)$ defined by $\phi(\alpha) =\rho^\alpha$ is a semigroup homomorphism by Proposition \ref{p22}. Using Proposition \ref{p1}, the map $\phi$ is onto. Now we need to show that $\phi$ is injective. For, let $\alpha,\beta \in OX_n$ such that $\rho^\alpha =\rho^\beta.$ For any $x\in X_n,$ $\rho^\alpha (Sx)=\rho(e_x, e_x\alpha, e_A)$ where $e_A \;\mathscr L\; \alpha \text{ and } \rho^\beta (Sx)=\rho(e_x, e_x\beta, e_B),\; e_B\; \mathscr L\; \beta .$ Since $\rho^\alpha =\rho^\beta, $ we have $$\rho(e_x, e_x\alpha, e_B) =\rho(e_x, e_x\alpha, e_B).$$ By Lemma \ref{l1}, $e_x\alpha =e_x\beta.$ It follows that $ x\alpha =x\beta \text{ for all } x\in X_n \text{ and }\alpha=\beta.$ \end{proof}
	
	\subsection{Power set category}
	Let $X_n=\{1\leq 2\leq \cdots \leq n\}$ be a non empty finite chain and in order to avoid trivialities, assume that $n\geq 3.$ Given any finite chain, one can construct a very simple category $P_o(X_n)$ from $X_n$ with the object set as the set of all proper subsets of $X_n.$ Given any two subsets $A,B;$ a morphism from $A$
	to $B$ is an order-preserving transformation  from $A$ to $B.$ We will call $P_0(X_n)$ the power set category and it has a natural choice of subobjects; the one provided by set inclusions. That is we have an inclusion function $j=j_A^B: A\to B$ if $A\subseteq B.$ 
	In the following proposition we prove that $P_o(X_n)$ is a normal category.
	\begin{prop}
		The power set category $P_o(X_n)$ is a normal category.
	\end{prop}
	\begin{proof}
		It is easy to see that $P_o(X_n)$ is a category with subobjects and the subobject relation is induced by the usual subset relation. Given an inclusion $j_{A'}^A\text{ where } A'\subseteq A,$ define a retraction $e: A\to A'$ as follows:\\
		Let $A'=\{x_1\leq x_2\leq\cdots\leq x_k\}$ and $x_i\in X_n, \; i=1,2,\cdots,k.$ Define 
		\begin{equation}
			(x)e = \begin{cases}
				x & \text{ if } x\in A'\\
				x_i & \text{ if } x\notin A' \text{ and } x_i\leq x\leq x_{i+1}\; i,\in\{1,2,\cdots k-1\}.\\
				x_1 & \text{ if } x\leq x_1\\
				x_k & \text{ if } x\geq x_k
			\end{cases}		
		\end{equation}
		Then clearly, $e\in OX_n \text{ and }je =1_{A'}.$
		Given any morphism( order-preserving transformation) $f: A\to B;$ let $B'= Im\;f$ and $A'$ is the cross-section of the partition of $A$ determined by $ker\;f.$ Then $f$ has a normal factorization and $f=euj,$ where $u =f\vert_{A'}$ is a bijection and $j=j_{B'}^B.$ Given any $A\subseteq X_n$, let $\gamma$ be a cone in $P_o(X_n)$ with vertex $A$ is defined as follows. Let $ u: X_n\to A$ be an order-preserving transformation such that $u(a)= a\text{ for all }a\in A.$ For any $B\subseteq X_n,$ define $\gamma(B)= u\vert_B: B\to A.$ Then $\gamma$ is a normal cone with $\gamma(A) =1_A.$ Thus $P_o(X_n)$ is a normal category.
	\end{proof}
	In the following theorem we prove that the categories $P_o(X_n)$ and $\mathscr L(X_n)$ are isomorphic. For that, we show that there exists an inclusion preserving functor from $\mathscr L(OX_n)$ to $P_o(X_n)$ which is an order isomorphism, $v$-injective, $v$-surjective and fully-faithful.
	\begin{thm}
		
		The categories $P_o(X_n)$ and $\mathscr L(OX_n)$ are isomorphic.
	\end{thm}
	\begin{proof}
		Define a functor $F:\mathscr L(OX_n)\to P_o(X_n) $ as follows: For $Se_A\in v\mathscr L(OX_n)$ and a morphism $\rho(e_A,u,e_B)\in \mathscr L(OX_n)$ we have $$ vF(Se_A) =A\quad and \quad F(\rho(e_A,u,e_B))= u\vert_A.$$ Clearly $F$ is well defined by Lemma \ref{le1} and Lemma \ref{l1}. Now let $\rho(e_A,u,e_B),\rho(e_B,v,e_C)$ be two composable morphisms in $\mathscr L(OX_n)$. Then $$ \rho(e_A,u,e_B)\rho(e_B,v,e_C)=\rho(e_A,uv,e_C). $$ Now $F(\rho(e_A,uv,e_C))= uv\vert_A= u\vert_Av\vert_B= F(\rho(e_A,u,e_B))F(\rho(e_B,v,e_C)).$ Hence $F$ is a functor.  Using the Lemma \ref{le1} it is easy to prove that $F$ is inclusion preserving and $vF$ is an order isomorphism.\par
		Now we prove that $vF$ is a bijection. For, Let $A\subseteq X_n$ such that $A=\{x_1\leq x_2\leq\cdots\leq x_k\}.$ Then define \begin{equation}
			e(x) = \begin{cases}
				x_1 & \text{ if } x\leq x_1\\
				x_i & \text{ if }  x_{i-1}< x\leq x_{i}\; i\in\{2,3,\cdots k\}.\\
				x_k & \text{ if } x\geq x_k
			\end{cases}		
		\end{equation}
		Clearly $e$ is an idempotent order-preserving transformations  with $Im\; e=A.$ Now $F(Se) = Im\; e= A.$ Hence $vF$ is $v$-surjective. By Lemma \ref{le1} it follows that $vF$ is injective. To complete the proof only need to prove $F$ is fully-faithful. Now let $f$ be an order-preserving transformation from $A$ to $B.$ Then $e_Af$ is an order-preserving transformation with image contained in $B$ and $ e_Af\vert_{A} =f.$    So $e_Af\in e_ASe_B$ and $\rho(e_A,e_Af,e_B): Se_A\to Se_B$ such that $F(\rho(e_A,e_Af,e_B))= f.$ Hence $F$ is full. The proof of $F$ is faithfull follows from Lemma \ref{l1}. Hence the Theorem.
	\end{proof}
	Since the category $P_o(X_n)$ is isomorphic to $\mathscr L(OX_n)$, the corresponding semigroups of normal cones $T\mathscr L(OX_n)$ and $TP_o(X_n)$ are isomorphic. But using Theorem \ref{t1} we get $TP_o(X_n)$ is isomorphic to $OX_n.$ Summarising, we have the following theorem.
	\begin{thm}
		$TP_o(X_n)$ is isomorphic to the semigroup $OX_n$ of non-invertible order-preserving transformation on a finite chain $X_n.$
	\end{thm} 
	\begin{rmk}
		All normal cones in $P_o(X_n)$ can be described as follows. Let $\gamma$ be a normal cone in $P_o(X_n)$ with vertex $A\subseteq X.$ Then let $\alpha: X_n\to X_n$ be defined as follows. $$ (x)\alpha = (x)\gamma(\{x\}),\text{ for all } x\in X_n.$$ Then arguing on similar lines as in the proof of Proposition \ref{p1}, we can see that  $\alpha\in OX_n$ and $\gamma =\rho^\alpha.$ Observe that $TP_o(X_n)$ is a representation of the semigroup $OX_n$; wherein an element  in $OX_n$ is represented as a normal cone in $P_o(X_n),$ which is nothing but a generalization of normal mapping
		of Grillet\cite{pag}.
	\end{rmk}
	\section{The Category of Partitions of a set}\label{s5}
	
	A partition $\pi$ of  a set $X$ is a family of subsets $A_i$ of $X$ such that $\cup A_i= X$ and $A_i\cap A_j=\phi\text{ for }i\neq j.$ A partition is said to be non-identity if at least one $A_i $ has more than one element. Since $X_n=\{1\leq2\leq\cdots\leq n\}$ is a partially ordered set, there are partitions of $X_n$ for which each $A_i$ is an interval. We call such partitions of $X_n$ as ordered partitions and they can be regarded as a finite chain of length $m\leq n$ as follows.\par
	Let $\pi=\cup_{i=1}^{m}A_i$ be an ordered partition. Then
	 $A_1$ consists of first $k_1$ elements of $X_n,\text{ say, } 1,2,\cdots k_1$ and $A_2$ consists of next $k_2$ elements, $k_1+1, k_1+2,\cdots k_1+k_2$ and so on.  Thus each $A_i$ consists of $k_i$ elements and $k_1+k_2+\cdots +k_m= n.$ For this reason, we deonte a non-identity ordered partition $\pi$ of $X_n$ as $\pi=\{A_1\leq A_2\leq \cdots \leq A_m: m<n\}$ and $\pi= \cup_{i=1}^m A_i. $
	\par In this article we consider only non-identity ordered partitions of $X_n$ so that we can consider them as partially ordered sets. Given a non-identity partition $\pi$ of $X_n$, we denote by $\bar\pi$ the set of all order-preserving transformations  from $\pi$ to $X_n.$ Let $\pi_1,\pi_2$ be two ordered partitions of $X_n$ and let $\eta$ be an order-preserving transformation from $\pi_2\text{ to } \pi_1.$ We define $\eta^*:\bar\pi_1\to \bar\pi_2$ by $(\alpha)\eta^*=\eta\alpha\text{ for every }\alpha \in \bar\pi.$ Now the category of ordered partitions $\prod_o(X_n)$ of the chain $X_n$ has the vertex set 
	\begin{center}
		$ v\prod_o(X_n) =\{\bar \pi: \pi \text{ is a non-identity ordered partition of } X_n\}$
	\end{center}
	
	and a morphism in $ \prod_o(X_n)$ from $\bar{\pi_1}$ to $\bar{\pi_2}$ is given by $\eta^*$
	as defined above where $\eta: \pi_2\to \pi_1$. Observe that 
	$\eta: \pi_2\to \pi_1$ gives a morphism from $\bar{\pi_1}\to \bar{\pi_2};$ and it is precisely for this reason, the category of ordered partitions is defined as above.  
	
	\par If $\pi_1=\{A_i: i\in I\}$ and $\pi_2=\{B_j: j\in J\},$ define a partial order on $\prod_o(X_n)$ as $$ \bar{\pi_1}\leq \bar{\pi_2} \text{ if and only if } \pi_2 \leq \pi_1;\text{ where }$$ $$ \pi_2\leq\pi_1 \text{ if and only if for each }j, B_j\subseteq A_i\text{ for some } i. $$
	If $\pi_2\leq \pi_1,$ then the map $v: B_j\to A_i$ is the inclusion map from $\pi_2$ to $\pi_1$ and $v^*: \bar{\pi_1}\to \bar{\pi_2}$ is a morphism in $\prod_o(X_n).$ We consider $v^*$ as the inclusion morphism from $\bar{\pi_1} \text{ to }\bar{\pi_2}$. If $v^*= j^{\bar{\pi_2}}_{\bar{\pi_1}}: \bar{\pi_1}\to \bar{\pi_2}$ is the inclusion morphism, then there exists a retraction $ \zeta^*: \bar{\pi_2} \to \bar{\pi_1}$, $i.e.,$ $v^*\zeta^*= 1_{\bar{\pi_1}.}$ Here $\zeta: \pi_1\to \pi_2$ given by $(A_i)\zeta = B_j$ where $B_j\in \pi_2 $ chosen such that  $B_j\subseteq A_i $ so that $\zeta$ is a well defined order-preserving transformation. Thus every inclusion $v^*$ splits.\par
	
	Now if we define $P$ as the subcategory of $\prod_o(X)$ with $vP = v\prod_o(X)$ and morphisms as inclusion morphisms $v^*,$ then we can verify that $(\prod_o(X),P)$ forms a category with subobjects. The next theorem gives a normal
	factorization of a morphism in $\prod_o(X)$.
	\begin{thm}\label{t1}
		Let $\eta^*$ be a morphism from $\bar{\pi_1}$ to $\bar{\pi_2}$ in $\prod_o(X_n)$ with $\eta:\pi_2\to\pi_1$ . Then there exists ordered partitions $\pi_\sigma$ and $\pi_\gamma$  and $v:\pi_2\to\pi_\sigma, u: \pi_\sigma\to \pi_\gamma $ and $\zeta: \pi_\gamma\to\pi_1$ such that $\eta^*= \zeta^*u^*v^*$ such that $\zeta^*$ is a retraction, $u^*$ is an isomorphism and $v^*$ is an inclusion. 
	\end{thm}
	\begin{proof}
		For each ordered partition $\pi =\{C_1\leq C_2\leq \cdots\leq C_k: k\leq n\}$ and $x\in X_n,$ we denote by $[x]=[x]_\pi $ the interval $C_k\in \pi$ such that $x\in C_k.$ That is $[x]_\pi$ is the equivalence class of $x$ in $\pi,$ considering $\pi$ as an equivalence relation. Let $\sigma =\sigma_\eta$ be the equivalence relation on $X_n$ defined as follows. $$\sigma =\{(x,y): [x]_{\pi_2}\eta = [y]_{\pi_2}\eta\}.$$ Then as equivalence relations $\pi_2\subseteq \eta.$ Let $v:\pi_2\to \sigma$ be the inclusion given by $[x]_{\pi_2}\to [x]_\sigma$. Since $ v:\pi_2 \to \sigma $ is the  inclusion map, we see that  $v^*: \bar\sigma \to \bar{\pi_2}$ is the inclusion morphism. Let $\gamma =\gamma_\eta$ be the partition of $X_n$ defined as follows. 
		Let $\pi_1 =\{A_1\leq A_2\leq \cdots \leq A_k\}$. Then $Im\; \eta = \{A_i: i\in I\subseteq\{1,2,\cdots k\}\}.$ Let $Im\;\eta = \{A_{i_1}\leq A_{i_2}\leq \cdots \leq A_{i_m}\}.$ Then define 
		$$ \gamma = C_{i_1}\cup C_{i_2}\cup \cdots \cup C_{i_m}, \text{  where } $$

		\begin{equation}
			\begin{split}
				C_{i_1}= & A_1\cup A_2\cup \cdots \cup A_{i_1}\\
				C_{i_2} = & A_{i_1+1}\cup A_{i_1+2}\cup \cdots \cup A_{i_2} \\
				\vdots\qquad &\vdots\qquad\qquad\vdots\\
				C_{i_{m-1}}= & A_{i_{m-2}+1}\cup\cdots \cup A_{i_{m-1}} \\
				C_{i_m} =& A_{i_{m-1}+1}\cup \cdots \cup A_{i_m}\cup A_{i_m+1} \cup\cdots\cup A_k.  
			\end{split}		
		\end{equation}
		Then clearly $\gamma$ is an ordered partition of $X_n$ and $\pi_1\subseteq \gamma.$ Let $\zeta:\gamma\to \pi_1$ be defined as $(C_{i_j})\zeta =A_{i_j}$ for $j=1,2,\cdots,m.$ Then we see that $\zeta^*$ is a retraction from $\bar\pi_1$ to $\bar\gamma.$ Now define $u:\sigma \to \gamma$ as $[x]_\sigma u=[ [x]_{\pi_2}\eta]_\gamma.
		$ Clearly $u$ is well defined and is an order-preserving bijection. Since $u$ is a bijection $u^*:\bar\gamma\to \bar\sigma$ is also a bijection and so is an isomorphism in $\prod_o(X_n).$ To show that $\eta^*=\zeta^*u^*v^*$ for every $[x]_{\pi_2}\in \pi_2,$ $$ ([x]_{\pi_2})vu\zeta = ([x]_\sigma)u\zeta = [[x]_{\pi_2}\eta]_\gamma\zeta= [x]_{\pi_2}\eta.$$
		Hence, for any $\alpha\bar{\pi_1},$ 
		$$ (\alpha)\eta^*= (\alpha)(vu\zeta)^* =vu\zeta\alpha =(\alpha)\zeta^*u^*v^*.$$ Thus $\eta^*=\zeta^*u^*v^*$ and so every morphism in $\prod_o(X_n)$ has a normal factorization.
		\end{proof}
		\begin{thm}
			$\prod_o(X_n)$ is a normal category.
		\end{thm}
	\begin{proof}
		We have already seen that $\prod_o(X_n)$ is a category with subobjects and every inclusion splits.	Using Theorem \ref{t1}, every morphism  has normal factorization.  Thus to prove $\prod_o(X_n)$ is a normal category, it  suffices to show that, for each $ \bar\pi\in v\prod_o(X_n)$ there exists an idempotent normal cone $\gamma$ with $c_\gamma= \bar\pi.$ Let $\pi= \{A_1\leq A_2\leq\cdots A_k:k< n\}$ be an ordered partition of $X_n$ and $u: X_n\to X_n$ be an order-preserving map such that the image of $u$ is a cross section of $\pi$ and $u$ is a constant on each $A_i.$  For each ordered partition $\pi_\alpha=\{B_1\leq B_2\leq \cdots\leq B_m: m<n\}$, define $u_\alpha: \pi\to\pi_\alpha$ by $A_i\mapsto [(A_i)u]_{\pi_\alpha}$. Consider $\bar u: v\prod_o(X_n)\to \prod_o(X_n)$ defined by $\bar u(\bar{\pi_\alpha})= u_\alpha^*.$ Now we prove that $\bar u$ is an idempotent cone with vertex $\bar\pi.$ For any $\bar\pi_\alpha\in v\pi_o(X_n),\; \bar u( \bar\pi_\alpha)$ is an order- preserving transformation from $\bar\pi_\alpha\text{ to } \bar\pi.$ Hence, $\bar u( \bar\pi_\alpha)\in \prod_o(X_n)(\bar\pi_\alpha,\bar\pi ),\;\;\;\forall \bar\pi_\alpha \in v\prod_o(X_n).$ Let $\pi_0=\{C_1\leq C_2\leq \cdots \leq C_p: p<n\}$ be any ordered partition  with $\pi_\alpha \subseteq \pi_0$, then for each $B_i\in \pi_\alpha$ there exists some $C_r =[B_j]_{\pi_0}\in \pi_0,$ where $[B_j]_{\pi_0}$ denotes the equivalence class containing the interval $B_j$ under the partition $\pi_0.$ The map $v: \pi_\alpha\to \pi_0$ maps $B_j\mapsto [B_j]_{\pi_0}$. Then $u_\alpha v:\pi\to\pi_\alpha$  is given by  $(A_i)u_\alpha v= [(A_i)u]_{\pi_\alpha}v=[(A_i)u]_{\pi_0}.$ This is nothing but $(A_i)u_0.$ Thus $u_\alpha v=u_0 \text{ and } (u_\alpha v)^*= u_o^*$ and $ \bar u(\bar{\pi_0})= j_{\bar{\pi_0}}^{\bar{\pi_\alpha}}\bar u(\bar{\pi_\alpha}).$ Taking $\pi_\alpha=\pi$ in the definition of $u_\alpha, $ we see that $u_\alpha:A_i\to A_i$ so that $\bar u(\bar{\pi_\alpha})$ is the identity. Hence $\bar u$ will be an idempotent normal cone with vertex $\bar{\pi}.$ Thus $\prod_o(X_n)$ is a normal category.
		
	\end{proof}
	In the following we prove that the principal right ideal category $\mathscr R(OX_n)$ of the semigroup $OX_n$ is isomorphic to the category $\prod_o(X_n)$ of ordered partitions of $X_n.$\par
	The ordered partitions of $X_n$ can be identified with the idempotents in $OX_n$ as follows. Let $\pi$ be an ordered partition of $X_n$ and let $A$ be a cross-section of $\pi.$ Then $\pi=\{[a]: a\in A\}.$ Let $e: X_n\to X_n$ be defined by $(x)e =a$ where $a\in A, x\in [a]_\pi.$ Clearly $e$ is an idempotent in $OX_n$ such that $\pi_e=\pi\text{ where }\pi_e\text{ is the partition determined by the kernel of }e.$ Since $e$ is an order-preserving transformation $\pi_e$ will be an ordered partition. For any $\alpha,$ we denote by $\pi_\alpha$ the ordered partition of $X_n$ determined by the $ker\; \alpha.$ We may write $\pi_\alpha =\{(x,y): (x)\alpha=(y)\alpha:\; x,y\in X_n\}$ or $\pi_\alpha=\{(x)\alpha^{-1}: x\in X_n\}.$ Given any $v\in fSe$ where $S=OX_n\; e,f\in E(OX_n)$ define $\eta_v: \pi_f\to\pi_e$ by $([x]_{\pi_f})\eta= ((x)v)e^{-1}.$ It can be shown that $\eta_v$ is well defined and order-preserving when regarding $\pi_f$ and $\pi_e$ as partially ordered sets. Then $\eta_v^*: \bar{\pi_e}\to \bar{\pi_f}$ is a morphism in $\prod_o(X_n).$ Now we define a  functor isomorphism $G:\mathscr R(OX_n)\to\prod_o(X_n) $ by $$ vG(eS)=\bar{\pi_e} \quad \text{and}\quad G(\lambda(e,v,f))= \eta_v^*.$$
	We conclude this section by identifying the principal right ideal category $\mathscr R(S)$ with partition category $\prod_o(X_n).$
	\begin{thm}
		$\mathscr R(OX_n)$ is isomorphic to $\prod_o(X_n)$ as normal categories.
	\end{thm}
	The above theorem can be proved by imitating the same proof technique as that of Theorem 18. of \cite{np}. 
	
	\begin{prop}
		The semigroup $OX_n$ is isomorphic to a subsemigroup of $T\mathscr R(OX_n)$.
	\end{prop}
	\begin{proof}
	Consider the map $\phi: OX_n\to T\mathscr R(OX_n) $ given by $\alpha\mapsto \lambda^\alpha.$ Then $\phi$ is a homomorphism by  dual of the Proposition $\ref{p22}.$  Now we need to show that $\phi$ is faithful. For, let $\alpha,\beta\in OX_n$ such that $\alpha\neq\beta $ and $\alpha\mathscr R\beta.$ Since  $\alpha\mathscr R\beta$ we have $ker\; \alpha=ker\;\beta.$ Suppose $ Im\;\alpha=\{x_1\leq x_2\leq\cdots\leq x_k\}$ and $Im\;\beta=\{y_1\leq y_2\leq\cdots\leq y_k\}$,  each $x_i$'s and each $y_i$'s are distinct.  The partition of $X_n$ determined by the kernel of $\alpha\text{ and } \beta$ be $\pi= A_1\leq A_2\leq \cdots\leq A_k.$ Then $(A_i)\alpha=x_i$ and $(A_i)\beta =y_i \text{ for } i=1,2,\cdots k.$
	Let $i$ be such that $x_i\neq y_i$ and $x_j=y_j \;\forall \;j<i.$ Then $x_i<y_i$ or $y_i< x_i.$	
	 Without loss of generality  let $x_i<y_i.$  Using this $x_i$ we define an idempotent transformation $e\in OX_n$ such that $\lambda^\alpha(eS)\neq \lambda^\beta(eS). $ Define $e: OX_n\to OX_n$ by 
	\begin{equation}
		e(x) = \begin{cases}
				x_i & \text{ if }   x\leq x_{i}.\\
			x_k & \text{ if } x\geq x_i+1.
		\end{cases}		
	\end{equation}
We can observe that $\lambda^\alpha(eS)= \lambda^\beta(eS)$ if and only if $\alpha e=\beta e$ which is not possible. Thus $\lambda^\alpha\neq \lambda^\beta $ whenever $\alpha\neq\beta$ and $\phi$ is a faithful representation of $OX_n$ to $T\mathscr R(OX_n).$
	\end{proof}


\begin{thebibliography}{99}
		
		\bibitem{com}P. A. Azeef Muhammed and A. R. Rajan. Cross-connections of completely simple semigroups.Asian-European J. Math., 09(03):1650053, 2016.
		\bibitem{pa} P. A. Azeef Muhammed and A. R. Rajan. Cross-connections of the singular transformation semigroup. J.
		Algebra Appl., 2017. DOI: 10.14232/actasm-017-044-z
		
		\bibitem{am} P.A. Azeef Muhammed, “Cross-connections of linear transformation semigroups”, Semigroup Forum 97(3)
		(2018), 457-470.
		\bibitem{pz} P.A. Azeef Muhammed, “Cross-connections and variants of the full transformation semigroup”, Acta Sci.
		Math. (Szeged) 84(3-4) (2018), 377-399.
	%	\bibitem{aa} P.A. Azeef Muhammed and A.R. Rajan, “Cross-connections of completely simple semigroups”, Asian European J. Math. 09(03) (2016), 1650053.
		\bibitem{pgr} P.A. Azeef Muhammed, P. G. Romeo, K. S. S. Nambooripad. Cross connection structure of concordant semigroups.
		\bibitem{az}P. A. Azeef Muhammed and M. V. Volkov. Inductive groupoids and cross-connections of regular semi-
		groups. Acta Math. Hungar., 2018. DOI: 10(1):262–268, 1975.
		\bibitem{ah}A.H.Clifford and G.B.Preston, The algebraic theory of semigroups vol 1 , Math
		Surveys No.7, American Mathematical Society, Providence,R.I.(1961)
		\bibitem{easdown} D. Easdown, Biordered sets come from semigroups, J. Algebra 96 (1985), 581–591.
		\bibitem{gracinda}  Gomes, G.M.S., Howie, J.M.: On the ranks of certain semigroups of order-preserving transformations.
		Semigroup Forum 45, 272–282 (1992)
		\bibitem{pag}P. A. Grillet. Structure of regular semigroups: Cross-connections. Semigroup Forum, 8:254–259, 1974.
		\bibitem{hall}Thomas Eric Hall. On regular semigroups. Journal of algebra, 24(1):1–24, 1973.
		\bibitem{jmh}J.M.Howie, Fundamentals of semigroup theory, Clarendon Press, Oxford. ISBN 0-19-851194-9(1995).
		\bibitem{ams}K. S. S. Nambooripad, Structure of regular semigroups. I, Mem. Amer. Math. Soc. 22 (1979), no.
		224.
		
		\bibitem{kss} K. S. S. Nambooripad. Structure of Regular Semigroups. II. Cross-connections. Publication No. 15. Centre
		for Mathematical Sciences, Thiruvananthapuram, 1989.
		\bibitem{kn} K. S. S. Nambooripad. Theory of Cross-connections. Publication No. 28. Centre for Mathematical Sciences,
		Thiruvananthapuram, 1994.
		%\bibitem{ksn} K. S. S. Nambooripad. Cross-connections. In Proceedings of the International Symposium on Semigroups
		%and Applications, pages 1–25, Thiruvananthapuram, 2007. University of Kerala.
		\bibitem{kk}K.S.S. Nambooripad : Theory of Regular Semigroups, Sayahna Foundation Trivandrum,
		2018.
		\bibitem{Mac} Saunders Mac Lane: Categories for the Working Mathematician, 
		Second edition, 0-387-9803-8, Springer-Verlag New york, Berlin Heidelberg Inc., 
		1998.
		\bibitem{np} A. R. Rajan and P. A. Azeef Muhamed. Normal category of partitions of a set.  Southeast Asian Bulletin of Mathematics · September 2015
	\end{thebibliography}
\end{document}